\documentclass[a4paper,12pt]{article}

\usepackage{amsmath}
\usepackage{amsfonts} 
\usepackage{amssymb,amsthm}
\usepackage{a4wide}
\usepackage[english]{babel}   

\newtheorem{de}{Definition}[section]
\newtheorem{theo}[de]{Theorem}
\newtheorem{prop}[de]{Proposition}

\newtheorem{lem}[de]{Lemma}

\begin{document}

\title{Non-split almost complex and \\ non-split Riemannian supermanifolds}
\author{
{\bf Matthias Kalus}\\
{\small  Fakult\"at f\"ur Mathematik}\\{\small Ruhr-Universit\"at Bochum}\\ {\small D-44780 Bochum, Germany}\\
}
\date{}
\maketitle
\centerline{ {\bf Keywords: } supermanifold; almost complex structure; Riemannian metric; non-split}
\centerline{  {\bf MSC2010:}  32Q60, 53C20, 58A50}
\renewcommand{\thefootnote}{\arabic{footnote}}

\begin{abstract}\noindent 
Non-split almost complex supermanifolds and non-split Riemannian supermanifolds are studied. The first obstacle for a  splitting is parametrized by group orbits on an infinite dimensional vector space. Further it is shown that non-split structures appear in the first case as deformations of a split reduction and in the second case as the deformation of an  underlying metric. 
In contrast to non-split deformations of complex supermanifolds, these deformations can be restricted by cut-off functions to local deformations. A class of examples of nowhere split structures constructed from almost complex manifolds of dimension $6$ and higher, is provided for both cases.
\end{abstract}

Even almost complex structures and Riemannian metrics define global tensor fields on real supermanifolds. Denoting a real supermanifold by $\mathcal M=(M,\mathcal C_\mathcal M^\infty)$ and the global super vector fields on $\mathcal M$ by $\mathcal V_\mathcal M$, the tensors lie in $End(\mathcal V_\mathcal M)_{\bar 0}$, resp. $Hom(\mathcal V_\mathcal M,\mathcal V_\mathcal M^\ast)_{\bar 0}$. Fixing a Batchelor model $\mathcal M \to (M,\Gamma_{\Lambda E^\ast}^\infty)$, the $\mathbb Z$-degree zero part $J_R$ of an even  almost complex structure $J  \in End(\mathcal V_\mathcal M)_{\bar 0}$ is again an almost complex structure on $\mathcal M$. This raises the question, whether there is a Batchelor model, such that $J$ equals its reduction $J_R$. Or equivalently, denoting by $\mathcal N$ the nilpotent superfunctions: if the exact sequence $$0 \to \mathcal N^2 \to \mathcal C^\infty_\mathcal M \to \mathcal C^\infty_M\oplus \Gamma_{E^\ast}^\infty\to 0$$ is split with  an $\alpha:\mathcal C^\infty_M\oplus \Gamma_{E^\ast}^\infty \to  \mathcal C^\infty_\mathcal M$ such that the induced automorphism $\hat \alpha:\mathcal C_\mathcal M^\infty \to \mathcal C_\mathcal M^\infty$ transforms $J_R$ to $J$.  In the case of a positive answer we call the tensor split. The analogue question can be formulated for even Riemannian metrics where the reduction $g_R$ of a metric $g \in Hom(\mathcal V_\mathcal M,\mathcal V_\mathcal M^\ast)_{\bar 0}$ is given by the Riemannian metric $g_0+g_2$.

\bigskip\noindent
For complex structures being integrable almost complex structures, existence of a splitting was studied in \cite{Gr}, \cite{Ro1} and \cite{Va}: there exist non-split complex supermanifolds, all of them being deformations of split complex supermanifolds. The parameter spaces of deformations are given by orbits of the automorphism group of the associated Batchelor bundle on a certain non-abelian first cohomology. Here the existence of local complex coordinates makes the splitting problem a problem of global cohomology. The splitting question for even symplectic supermanifolds was answered in \cite{Ro2} by identifying the symplectic supermanifold with an underlying symplectic manifold and a Batchelor bundle with metric and connection. It is shown that all terms of degree higher than 2 in a symplectic form can be erased by the  choice of a Batchelor model. Hence all symplectic supermanifolds are split in the above sense.

\bigskip\noindent
In this paper the existence of a splitting for even almost complex structures as well as even Riemannian metrics is studied. It is shown that all almost complex structures appear as deformations of split structures and all Riemannian metrics appear as deformations of underlying metrics. In both cases but in contrast to the complex case, these  deformations can be restricted by smooth cut-off functions to local deformations. For almost complex structures the splitting problem stated above can be expressed as: what is the obstacle for having local coordinates near any point such that the almost complex structure is represented by a purely numerical matrix. In the Riemannian case (similar to the symplectic case in \cite{Ro2}), the reduction is asked to be a purely numerical matrix on $\mathcal V_{\mathcal M,-1}^{\otimes 2}$ and to have matrix entries of degree less or equal to $2$ on the three remaining blocks of $(\mathcal V_{\mathcal M,0}\oplus \mathcal V_{\mathcal M,-1})^{\otimes 2}$. The first obstacle for a splitting is described for both problems. Finally explicit examples of non-split almost complex structures, resp. Riemannian metrics are given. The results and the applied methods are summarized in the following.

\bigskip\noindent
\textit{Contents.} In the first section an almost complex structure is decomposed via the finite log series into its reduction (the degree zero term) and its degree increasing term. With respect to these components the lowest degree obstacle for isomorphy of almost complex supermanifolds is deduced. Fixing the reduction, these obstacles are parametrized by the orbits of a quotient of the group of transformations that are almost holomorphic with respect to the reduction up to a certain degree, acting on a quotient of tensor spaces. 

\bigskip\noindent
The second section deals with Riemannian metrics in an analogue way producing  results analogous to those in the almost complex case. Here the isometries of the reduction play the role of the almost holomorphic transformations. However the more complicated action of the automorphism group of the supermanifold on a metric and the fact that the reduction has no pure degree, require an adjustment of the techniques. 

\bigskip\noindent
Finally the third section contains a class of non-split examples for almost complex structures and Riemannian metrics. These are constructed on the supermanifold of differential forms on an arbitrary almost complex manifold of dimension higher than $4$. Some basic facts of almost complex geometry and a method to construct super vector fields are applied. In the almost complex and in the Riemannian case, the constructed non-split tensors are nowhere split, i.e. at no point of the manifold the matrix elements of the respective tensors satisfy the split property mentioned above.

\section{Non-split almost complex supermanifolds}
\noindent
Let $(\mathcal M,J)$ be an almost complex supermanifold with sheaf of superfunctions $\mathcal C^\infty_\mathcal M$. Denote the $\mathcal C^\infty_\mathcal M(M)$-module of global superderivations of $\mathcal C^\infty_\mathcal M$ by $\mathcal V_{\mathcal M}=\mathcal V_{\mathcal M,\bar 0}\oplus \mathcal V_{\mathcal M,\bar 1}$. Furthermore fix a Batchelor model $\mathcal M \to (M,\Gamma_{\Lambda E^\ast}^\infty)$ yielding $\mathbb Z$-gradings (denoted by lower indexes) and filtrations (denoted by upper indexes in brackets) on $\mathcal C_\mathcal M^\infty$, $\mathcal V_{\mathcal M}$ and $End(\mathcal V_{\mathcal M})$, the last denoting $\mathcal C^\infty_\mathcal M(M)$-linear maps. The even automorphism of $\mathcal C_\mathcal M^\infty(M)$-modules $J \in End(\mathcal V_{\mathcal M})_{\bar 0}$ can be uniquely decomposed into $J=J_R(Id+J_N)$ with invertible $J_R=J_0$ and nilpotent $J_N$. The finite exp and log series yield a unique representation $Id+J_N=\exp(Y)$ with $Y \in End^{(2)}(\mathcal V_{\mathcal M})_{\bar 0}$. 

\begin{lem}\label{le1}
 The tensor $J$ is an almost complex structure if and only if $J_R$ is an almost complex structure and $Y J_R+J_R Y=0$.
\end{lem}
\begin{proof}
 From  $J^2=-Id$ we obtain $J_R^2=-Id$ and $\exp(Y) J_R \exp(Y)=J_R$. For reasons of degree $Y_2 J_R+J_R Y_2=0$. Assume that $Y_{2k} J_R+J_R Y_{2k}=0$ holds for all $k< n$. Set $Y_{[2k]}:=\sum_{j=1}^kY_{2j}$. It is $\exp(Y)=\exp(Y_{[2n-2]})+{Y_{2n}}$ up to terms of degree $>2n$. Hence $\exp(Y) J_R \exp(Y)=\exp(Y_{[2n-2]})\exp(-Y_{[2n-2]})J_R+{Y_{2n} J_R+J_RY_{2n}}$ up to terms of degree  $>2n$. This completes the induction. The converse implication follows directly.
\end{proof}

We call $J_R$ the \textit{reduction} of $J$, deforming $J_R$ by $t\mapsto J_R\exp(tY)$. In particular $J_R$ yields an almost complex structure on $M$ and an almost complex structure on the vector bundle $E \to M$. Hence even and odd dimension of $\mathcal M$ are even. Further topological conditions on  $M$ and $E$ for the existence of an almost complex structure can be obtained from \cite{Mas} and e.g.~\cite{TE}. Adapted to our considerations the almost complex supermanifold $(\mathcal M,J)$ is {split} if there is a  Batchelor model, such that the almost complex structure $J$ has nilpotent component $Y=0$. Note that this problem is completely local since Lemma \ref{le1} allows cutting off the nilpotent $Y$ in $J=J_R\exp(Y)$.

\bigskip\noindent
Let $\Phi=(\varphi,\varphi^\ast)$ be an automorphism of the supermanifold $\mathcal M$. The global even isomorphism of superalgebras $\varphi^\ast \in Aut(\mathcal C_\mathcal M^\infty(M))_{\bar 0}$ over $\varphi$ is  decomposable into $\varphi^\ast=\exp(\zeta)\varphi_0^\ast $ with $\zeta \in \mathcal V_{\mathcal M,\bar 0}^{(2)}$ and  $\varphi_0^\ast$ preserving the $\mathbb Z$-degree induced by the Batchelor model (see e.g.~\cite{Ro1}). Denote by $Aut(E^\ast)$ the bundle automorphisms  over arbitrary diffeomorphisms of $M$, then $\varphi_0^\ast$ is induced by an element  $\varphi_0 \in Aut(E^\ast)$ over $\varphi$. 
The automorphism $\varphi^\ast$ transforms  $J$ into $\varphi^\ast.J$ given by  $(\varphi^\ast.J)(\chi):=\varphi^\ast(J((\varphi^\ast)^{-1}\chi \varphi^\ast))(\varphi^\ast)^{-1}$. 
Denoting $ad(\zeta):=[\zeta,\cdot]$, assuming $\zeta \in \mathcal V_{\mathcal M,\bar 0}^{(2k)}$ and applying $\varphi^\ast=(Id+\zeta)\varphi_0^\ast$ up to terms in $\mathcal V_{\mathcal M,\bar 0}^{(4k)}$, it is  
$\varphi^\ast.J=\varphi_0^\ast.J+[ad(\zeta),\varphi_0^\ast.J]$ up to terms in $End^{(4k)}(\mathcal V_{\mathcal M})_{\bar 0}$. 
Comparing both sides with respect to the degree yields:
 
\begin{prop}\label{lem2}
 The almost complex supermanifolds  $(\mathcal M,J)$ and $(\mathcal M,J^\prime)$ with structures $J=J_R \exp(Y)$, $J^\prime=J_R^\prime \exp(Y^\prime) $, $Y,Y^\prime \in  End^{(2k)}(\mathcal V_{\mathcal M})_{\bar 0}$ are isomorphic up to error terms in $End^{(4k)}(\mathcal V_{\mathcal M})_{\bar 0}$ via an automorphism $\varphi^\ast$ with $\varphi^\ast(\varphi_0^\ast)^{-1} \in \exp(\mathcal V_{\mathcal M,\bar 0}^{(2k)})$ if and only if there exist $\varphi_0 \in Aut(E^\ast)$ and  $\zeta \in  \mathcal V_{\mathcal M,\bar 0}^{(2k)}$ such that $J_R^\prime=\varphi_0^\ast.J_R$ and: 
 \begin{align*}
  Y_{2j}^\prime=\varphi_0^\ast.Y_{2j}-ad(\zeta_{2j})-J_R^\prime ad(\zeta_{2j})J_R^\prime, \quad k\leq j<2k
 \end{align*}
\end{prop}

From now on we fix the reduction $J_R$ and hence assume that for an automorphism $\psi^\ast$ of $\mathcal M$, the map  $\psi_0^\ast$ is  pseudo-holomorphic with respect to $J_R$, denoted $\psi_0^\ast\in Hol(\mathcal M,J_R)$. Let  $Hol(\mathcal M,J_R,2k)$ be the automorphisms $\psi^\ast=\exp(\xi)\psi^\ast_0$ of $\mathcal M$ such that  $J_R=\psi^\ast.J_R$ up to terms in $End^{(2k)}(\mathcal V_\mathcal M)_{\bar 0}$. Note that $\psi^\ast\in Hol(\mathcal M,J_R,2k)$ includes  $\psi_0^\ast \in Hol(\mathcal M,J_R)$ and that $\exp(\mathcal V_{\mathcal M,\bar 0}^{(2k)})\subset Hol(\mathcal M,J_R,2k)$ is a normal subgroup.

\bigskip\noindent
Define on the endomorphisms of real vector spaces $End_{\mathbb R}(\mathcal V_{\mathcal M})$ the $\mathcal C^\infty_\mathcal M(M)$-linear $\mathbb Z$-degree preserving map: $$F_{J_R}:End_{\mathbb R}(\mathcal V_{\mathcal M})\to  End_{\mathbb R}(\mathcal V_{\mathcal M}), \quad F_{J_R}(\gamma):=\gamma+J_R \gamma J_R$$
The set $F_{J_R}(End^{(2k)}(\mathcal V_{\mathcal M}))$  is by Lemma \ref{le1} exactly the nilpotent parts $Y$ of almost complex structures $J=J_R\exp(Y)$ deforming $J_R$ in degree $2k$ and higher. Note further that $F_{J_R} (ad(\mathcal V_{\mathcal M})) \subset End(\mathcal V_{\mathcal M})$ and more precisely $F_{J_R} (ad(\mathcal V_{\mathcal M,\bar 0}^{(2k)})) \subset End^{(2k)}(\mathcal V_{\mathcal M})_{\bar 0}$. 

\begin{de} Let the upper index $2k\in 2\mathbb N$ in curly brackets denote the sum of terms of $\mathbb Z$-degree $2k$ up to $4k-2$. For $J=J_R\exp(Y)$, $Y\in  End^{(2k)}(\mathcal V_{\mathcal M})_{\bar 0}$ we call the  class $[Y^{\{2k\}}]$ in the quotient of vector spaces $F_{J_R} ( End^{\{2k\}}(\mathcal V_{\mathcal M})_{\bar 0})/F_{J_R} (ad(\mathcal V_{\mathcal M,\bar 0}^{\{2k\}}))$ the $2k$-th split obstruction class of $J$. 
\end{de}

The $Hol(\mathcal M,J_R,2k)$-action on $F_{J_R} ( End^{(2k)}(\mathcal V_{\mathcal M})_{\bar 0})$ is given  up to terms in $End^{(4k)}(\mathcal V_\mathcal M)_{\bar 0}$ by $(\psi^\ast,Y) \mapsto J_R(J_R-\psi^\ast.J_R)+\psi^\ast.Y$. Since $\psi^\ast.F_{J_R} (ad(\mathcal V_{\mathcal M,\bar 0}^{\{2k\}}))\subset F_{J_R} (ad(\mathcal V_{\mathcal M,\bar 0}^{\{2k\}}))$, it is  well-defined on $F_{J_R} ( End^{\{2k\}}(\mathcal V_{\mathcal M})_{\bar 0})/F_{J_R} (ad(\mathcal V_{\mathcal M,\bar 0}^{\{2k\}}))$. By
 Proposition \ref{lem2} it induces an action of  $PHol(\mathcal M,J_R,2k):=Hol(\mathcal M,J_R,2k)/\exp(\mathcal V_{\mathcal M,\bar 0}^{(2k)})$ on  $F_{J_R} ( End^{\{2k\}}(\mathcal V_{\mathcal M})_{\bar 0})/F_{J_R} (ad(\mathcal V_{\mathcal M,\bar 0}^{\{2k\}}))$.
 It follows that for an almost complex supermanifold that is split up to terms of degree $2k$ and higher, the $2k$-th split obstruction class is well-defined up to the $PHol(\mathcal M,J_R,2k)$-action.  
Note that for a given almost complex structure $J=J_R\exp(Y)$ the obstructions can be checked starting with $j=1$ iteratively: if 
$Y_{2j}=ad(\zeta_{2j})+J_R ad(\zeta_{2j})J_R$ can be solved for a $\zeta_{2j}\in \mathcal V_{\mathcal M,2j}$ then there is an automorphism of the supermanifold $\mathcal M$ such that $J=J_R \exp(Y^\prime)$ with  $Y^\prime\in  End^{(2(j+1))}(\mathcal V_{\mathcal M})_{\bar 0}$. In the non-split case this procedure ends with a well-defined $2k$ and associated orbit of $2k$-th split obstruction classes. We note as a special case:

\begin{prop}
 Let $(\mathcal M,J_R)$ be a split almost complex supermanifold of odd dimension $2(2m+r)$, $m\geq 0$, $r \in \{0,1\}$. The almost complex supermanifolds $(\mathcal M,J)$ with reduction $J_R$ that are split up to terms of degree $(2m+r)+1$ and higher, correspond bijectively to the $PHol(\mathcal M,J_R,2(m+1))$-orbits on  $F_{J_R} ( End^{(2(m+1))}(\mathcal V_{\mathcal M})_{\bar 0})/F_{J_R} (ad(\mathcal V_{\mathcal M,\bar 0}^{(2(m+1))}))$.
\end{prop}

As a technical tool we note an identification for the quotient appearing in the split obstruction classes. Denote by $\mathcal E^1_\mathcal M=\mathcal E^1_{\mathcal M,\bar 0}\oplus \mathcal E^1_{\mathcal M,\bar 1}$ the global super-1-forms on $\mathcal M$ and by $\mbox{d}_\mathcal M$ the de Rham operator on the algebra $\mathcal E_\mathcal M$ of superforms. It is $End(\mathcal V_\mathcal M)=\mathcal V_\mathcal M \otimes_{\mathcal C^\infty_\mathcal M(M)} \mathcal E^{1}_\mathcal M $.

\begin{prop} \label{prop1}
For all $k$, the map
\begin{align*}
  \Theta_{J_R}:\ F_{J_R}\big(\mathcal V_\mathcal M \otimes_{\mathcal C^\infty_\mathcal M(M)}  \mathcal E^{1}_\mathcal M   \big)^{(2k)}_{\bar 0} & \longrightarrow  F_{J_R} ( End^{(2k)}(\mathcal V_{\mathcal M})_{\bar 0})/F_{J_R} (ad(\mathcal V_{\mathcal M,\bar 0}^{(2k)})) 
  \end{align*}
 locally for homogeneous arguments defined by  $$F_{J_R}(\chi\otimes {\rm d}_\mathcal M f) \longmapsto (-1)^{|f||\chi|}f\cdot F_{J_R}( ad(\chi))+F_{J_R} (ad(\mathcal V_{\mathcal M,\bar 0}^{(2k)}))$$
is a well-defined, surjective morphism of $\mathbb Z$-filtered super vector spaces. For any element $\psi^\ast \in Hol(\mathcal M,J_R,2k)$ and $[\psi^\ast] \in PHol(\mathcal M,J_R,2k)$ it is $\Theta_{J_R}(\psi^\ast.Z)=[\psi^\ast].(\Theta_{J_R}(Z))$. 
\end{prop}
\begin{proof}
For homogeneous components of $\chi\otimes  {\rm d}_\mathcal M f \in \mathcal V_\mathcal M \otimes_{\mathcal C^\infty_\mathcal M(M)}  \mathcal E^{1}_\mathcal M$ the decomposition 
 $\chi\otimes  {\rm d}_\mathcal M f=(-1)^{|f||\chi|}(f\cdot ad(\chi)-ad(f\chi))$  
is well-defined up to terms in $ad(\mathcal V_{\mathcal M})$. 
\end{proof}

\section{Non-Split Riemannian supermanifolds}
\noindent
Let $(\mathcal M,g)$ be a Riemannian supermanifold with even non-degenerate supersymmetric  form $g\in Hom(\mathcal V_\mathcal M \otimes_{\mathcal C^\infty_\mathcal M(M)} \mathcal V_\mathcal M,\mathcal C^\infty_\mathcal M(M))_{\bar 0}$. Here we will mostly regard $g$ as an isomorphism of $\mathcal C^\infty_\mathcal M$-modules $g \in Hom(\mathcal V_\mathcal M, \mathcal V_\mathcal M^\ast)_{\bar 0}$ with  $g(X)(Y)=(-1)^{|X||Y|}g(Y)(X)$ for homogeneous arguments. The context will fix which point of view is used. For a given  Batchelor model  $\mathcal M \to (M,\Gamma_{\Lambda E^\ast}^\infty)$ decompose $g=g_R(Id+g_N)$ with invertible $g_R=g_0+g_2$ and nilpotent $g_N\in End^{(2)}(\mathcal V_\mathcal M)_{\bar 0}$ such that $g_0g_N \in  Hom^{(4)}(\mathcal V_\mathcal M, \mathcal V_\mathcal M^\ast)_{\bar 0}$. With the finite log and exp series we write $g=g_R\exp(W)$ with $W \in End^{(2)}_{g_0}(\mathcal V_\mathcal M)_{\bar 0}$, where $End^{(2k)}_{g_0}(\mathcal V_\mathcal M)_{\bar 0}$ denotes those $W \in End^{(2k)}(\mathcal V_\mathcal M)_{\bar 0}$ such that $g_0W \in  Hom^{(2k+2)}(\mathcal V_\mathcal M, \mathcal V_\mathcal M^\ast)_{\bar 0}$.
\begin{lem}\label{le2}
 If the tensor $g=g_R\exp(W)$, $W \in End^{(2k)}_{g_0}(\mathcal V_\mathcal M)_{\bar 0}$ is a Riemannian metric then $g_R$ is a Riemannian metric and  $g_R(W(\cdot),\cdot)=g_R(\cdot,W(\cdot))$ up to terms of degree $4k+2$ and higher.
\end{lem}
\begin{proof} 
Due to supersymmetry $g_R(\exp(W)(\cdot),\cdot)=g_R(\cdot,\exp(W)(\cdot))$. The approximation $\exp(W)=1+W$ holds up to terms of degree $4k$ with error term $\frac{1}{2}W_{2k}^2$ in degree $4k$. Since $W \in End^{(2k)}_{g_0}(\mathcal V_\mathcal M)_{\bar 0}$ it is $g_RW_{2k}^2 \in  Hom^{(4k+2)}(\mathcal V_\mathcal M, \mathcal V_\mathcal M^\ast)_{\bar 0}$.
\end{proof}

We call $g_R$ the \textit{reduction} of $g$. Here the metric $g$ appears as a deformation of the underlying Riemannian metric $g_{0}$ on $M$ via $t\mapsto (g_{0}+t\cdot g_{2})\exp(\sum_{j=1}^\infty t^jW_{2j})$.  
Note that $g_R$ also yields a non-degenerate alternating form on the bundle $E$. So in contrast to the non-graded case there is a true condition for the existence of a Riemannian metric: the existence of a nowhere vanishing section of $E\wedge E\to M$. In particular the odd dimension of $\mathcal M$ has to be even. A Riemannian supermanifold $(\mathcal M,g)$ is {split}, if there is a  Batchelor model, such that the  Riemannian metric $g$ has nilpotent component $W=0$. Again the appearing deformations are essentially local via cutting off $g_R\exp(W)$ by $(g_{0}+f\cdot g_{2})\exp(\sum_{j=1}^\infty f^jW_{2j})$ with cut-off function $f$.

\bigskip\noindent
As before let $\Phi=(\varphi,\varphi^\ast)$, $\varphi^\ast=\exp(\zeta)\varphi_0^\ast$  be an automorphism of the supermanifold $\mathcal M$.  
We obtain $\varphi^\ast.g$ given by $(\varphi^\ast.g)(\chi)(\chi^\prime)=\varphi^\ast\big( g((\varphi^\ast)^{-1}\chi \varphi^\ast,(\varphi^\ast)^{-1}\chi^\prime \varphi^\ast)\big)$.  Assuming $\zeta \in \mathcal V_{\mathcal M,\bar 0}^{(2k)}$ this yields: 
\begin{align}\label{asd}
 \varphi^\ast.g=\varphi_0^\ast.g-(\varphi_0^\ast.g)(ad(\zeta)\otimes Id + Id \otimes ad(\zeta))+\zeta (\varphi_0^\ast.g)
\end{align}
in $Hom(\mathcal V_\mathcal M, \mathcal V_\mathcal M^\ast)_{\bar 0}$ up to terms in $Hom^{(4k)}(\mathcal V_\mathcal M, \mathcal V_\mathcal M^\ast)_{\bar 0}$. Note that for the term $\zeta (\varphi_0^\ast.g)$, the metric is regarded as an element in  $Hom(\mathcal V_\mathcal M\otimes_{\mathcal C^\infty_\mathcal M(M)}\mathcal V_\mathcal M ,\mathcal C^\infty_\mathcal M(M))_{\bar 0}$. Define $\mathcal V_{\mathcal M,g_0,\bar 0}^{(2k)}$ to be the elements in $\zeta \in \mathcal V_{\mathcal M,\bar 0}^{(2k)}$ satisfying $g_0(ad(\zeta)\otimes Id + Id \otimes ad(\zeta))+\zeta g_0\in Hom^{(2k+2)}(\mathcal V_\mathcal M,\mathcal V_\mathcal M^\ast)_{\bar 0}$. Comparing the terms in (\ref{asd}) with respect to the degree  yields:
 
\begin{prop}\label{lem3}
 The Riemannian supermanifolds  $(\mathcal M,g)$ and $(\mathcal M,g^\prime)$ with Riemannian metrics $g=g_R\exp(W)$, $g^\prime=g_R^\prime\exp(W^\prime)$, $W \in End^{(2k)}_{g_0}(\mathcal V_\mathcal M)_{\bar 0}$, $W^\prime \in End^{(2k)}_{g_0^\prime}(\mathcal V_\mathcal M)_{\bar 0}$ are isomorphic up to error terms in $Hom^{(4k)}(\mathcal V_\mathcal M, \mathcal V_\mathcal M^\ast)_{\bar 0}$  via an automorphism $\varphi^\ast$ with $\varphi^\ast(\varphi_0^{\ast})^{-1}\in\exp(\mathcal V_{\mathcal M,g_0^\prime,\bar 0}^{(2k)})$   if and only if there exist $\varphi_0 \in Aut(E^\ast)$ and $\zeta \in  \mathcal V_{\mathcal M,g_0^\prime,\bar 0}^{(2k)}$ such that 
$ g_R^\prime=\varphi_0^\ast.g_R$ and:
\begin{align*}
  W_{2j}^\prime=\varphi_0^\ast.W_{2j}-ad(\zeta_{2j})-\Big((g_R^\prime)^{-1} (ad^\ast(\zeta)-\zeta) g_R^\prime\Big)_{2j}, \quad k\leq j<2k
\end{align*}
 Here $\varphi_0^\ast.W$ is defined by $(\varphi_0^\ast.W)(\chi):=\varphi^\ast_0(W((\varphi_0^\ast)^{-1}\chi \varphi_0^\ast))(\varphi_0^\ast)^{-1}$ and the homomorphism $ad^\ast:\mathcal V_\mathcal M \to End_\mathbb R(End_\mathbb R(\mathcal V_\mathcal M,\mathcal C^\infty_\mathcal M(M)))$ denotes the  representation dual to $ad$.
\end{prop}

Fix $g_R$ from now on and denote by $Iso(\mathcal M,g_R,2k+2)$ the automorphisms $\psi^\ast=\exp(\xi)\psi^\ast_0$ of $\mathcal M$ such that  $g_R=\psi^\ast.g_R$ up to a term $S:=g_R-\psi^\ast.g_R \in  Hom^{(2k+2)}(\mathcal V_\mathcal M,\mathcal V_\mathcal M^\ast)_{\bar 0}$.  Note that this forces $g_0g_R^{-1}S \in Hom^{(2k+2)}(\mathcal V_\mathcal M,\mathcal V_\mathcal M^\ast)_{\bar 0}$. Further $\exp(\mathcal V_{\mathcal M,g_0,\bar 0}^{(2k)})\subset Iso(\mathcal M,g_R,2k+2)$ is a normal subgroup.

\bigskip\noindent
Parallel to the analysis of the almost complex structures we define the maps
\begin{align*}
 &F_{g_R}: End(\mathcal V_\mathcal M) \to End(\mathcal V_\mathcal M),&& \ F_{g_R}(\gamma):=\gamma+g_R^{-1} \gamma^\ast g_R\\
 &G_{g_R}: \quad \mathcal V_\mathcal M \quad\quad\to End(\mathcal V_\mathcal M),&& \ G_{g_R}(\zeta):=ad(\zeta)+g_R^{-1} (ad^\ast(\zeta)-\zeta) g_R&&&&&&&&&
\end{align*}
denoting by $\gamma^\ast$ the induced element in $End(\mathcal V_\mathcal M^\ast)$ and $ad^\ast$ as above. By Lemma \ref{le2} the elements in $F_{g_R}(End^{(2k)}_{g_0}(\mathcal V_\mathcal M)_{\bar 0})$  are up to degree $\geq 4k+2$ the appearing $W$s in Riemannian metrics $g=g_R\exp(W)$ that are split up to degree $\geq 2k$. Further $G_{g_R}(\mathcal V_{\mathcal M,g_0,\bar 0}^{(2k)})$ lies in $F_{g_R}(End^{(2k)}_{g_0}(\mathcal V_\mathcal M)_{\bar 0})$. 

\begin{de}  For $g=g_R\exp(W)$ with $W\in  End^{(2k)}(\mathcal V_{\mathcal M})_{\bar 0}$ we call the  class $[W^{\{2k\}}]$ in the quotient of vector spaces $F_{g_R} ( End^{(2k)}_{g_0}(\mathcal V_{\mathcal M})_{\bar 0})^{\{2k\}}/G_{g_R} (\mathcal V_{\mathcal M,g_0,\bar 0}^{(2k)})^{\{2k\}}$ the $2k$-th split obstruction class of $g$. 
\end{de}

The $Iso(\mathcal M,g_R,2k+2)$-action on  $F_{g_R} ( End^{(2k)}_{g_0}(\mathcal V_{\mathcal M})_{\bar 0})$ is given up to terms in $End^{(4k)}(\mathcal V_\mathcal M)_{\bar 0}$ by $(\psi^\ast,W)\mapsto  g_R^{-1}(\psi^\ast.g_R-g_R)+ \psi^\ast.W$. It is $\psi^\ast.G_{g_R} (\mathcal V_{\mathcal M,g_0,\bar 0}^{(2k)})\subset G_{g_R} (\mathcal V_{\mathcal M,g_0,\bar 0}^{(2k)})$ by direct calculation. Analog to the almost complex case using Proposition \ref{lem3}, the action of  $Iso(\mathcal M,g_R,2k+2)$ induces a $PIso(\mathcal M,g_R,2k+2):=Iso(\mathcal M,g_R,2k+2)/\exp(\mathcal V_{\mathcal M,g_0,\bar 0}^{(2k)})$-action on the quotient $F_{g_R} ( End^{(2k)}_{g_0}(\mathcal V_{\mathcal M})_{\bar 0})^{\{2k\}}/G_{g_R} (\mathcal V_{\mathcal M,g_0,\bar 0}^{(2k)})^{\{2k\}}$. Hence the $2k$-th split obstruction class is well-defined  up to the $PIso(\mathcal M,g_R,2k+2)$-action for a Riemannian supermanifold that is split up to terms of degree $2k+2$ and higher. We have in particular analogously to the almost complex case:
\begin{prop}
 Let $(\mathcal M,g_R)$ be a split Riemannian supermanifold of odd dimension $2(2m+r)$, $m\geq 0$, $r \in\{0,1\}$. The Riemannian supermanifolds $(\mathcal M,g)$ with reduction $g_R$ that are split up to terms of degree $(2m+r)+3$ and higher, correspond bijectively to the $Iso(\mathcal M,g_R,2(m+2))$-orbits on   $F_{g_R} ( End^{(2(m+1))}(\mathcal V_{\mathcal M})_{\bar 0})/G_{g_R} (\mathcal V_{\mathcal M,g_0,\bar 0}^{(2(m+1))})$.
\end{prop}

Also an analogy to Proposition \ref{prop1} holds:
\begin{prop} \label{prop2}
 The map
\begin{align*}
  \Theta_{g_R}:\ F_{g_R} ( End^{(2k)}_{g_0}(\mathcal V_{\mathcal M})_{\bar 0}) & \longrightarrow  F_{g_R} ( End^{(2k)}_{g_0}(\mathcal V_{\mathcal M})_{\bar 0})/G_{g_R} (\mathcal V_{\mathcal M,g_0,\bar 0}^{(2k)}) 
  \end{align*}
 locally defined by  $$F_{g_R}(\chi\otimes  {\rm d}_\mathcal M f) \longmapsto (-1)^{|f||\chi|}f \cdot G_{g_R}( \chi)+G_{g_R} (\mathcal V_{\mathcal M,g_0,\bar 0}^{(2k)})$$
is a well-defined surjective morphism of $\mathbb Z$-filtered vector spaces. For any element $\psi^\ast$ in $Iso(\mathcal M,g_R,2k+2)$ and $[\psi^\ast] \in PIso(\mathcal M,g_R,2k+2)$ it is $\Theta_{g_R}(\psi^\ast.Z)=[\psi^\ast].(\Theta_{g_R}(Z))$. 
\end{prop}
\begin{proof}
 Apply $F_{g_R}$ to  $\chi\otimes {\rm d}_\mathcal M f=(-1)^{(|f||\chi|}(f\cdot ad(\chi)-ad(f\chi))$ and add $f \chi-f\chi$ in the bracket. This yields a map $ F_{g_R} ( End^{(2k)}_{g_0}(\mathcal V_{\mathcal M})_{\bar 0}) \to  F_{g_R} ( End^{(2k)}(\mathcal V_{\mathcal M})_{\bar 0})/G_{g_R} (\mathcal V_{\mathcal M,\bar 0}^{(2k)}) $. Since  $\chi\otimes {\rm d}_\mathcal M f$ is in $ End^{(2k)}_{g_0}(\mathcal V_{\mathcal M})_{\bar 0}$, its degree $2k$ term is of the form $\sum \tilde f_i \frac{\partial}{\partial \xi_i}\otimes d_\mathcal M \hat f_i$ for an odd coordinate system $(\xi_i)$. This forces $f \chi \in \mathcal V_{\mathcal M,g_0,\bar 0}^{(2k)}$ by direct calculation.
\end{proof}

\section{Examples of global nowhere split structures}\label{sec3}
Here explicit examples of non-split almost complex structures, resp. non-split Riemannian metrics are given. The constructed tensors are nowhere split.

\bigskip\noindent
Let $(M, J_M)$ be an almost complex manifold of dimension $2n$ and let $\mathcal M$ be the supermanifold defined by differential forms, i.e.~$\mathcal C^\infty_\mathcal M=\mathcal E_M$. The vector fields in $\mathcal V_M$ act on $\mathcal C^\infty_\mathcal M$ by Lie derivation. Let further $\pi:\mathcal V_\mathcal M \to \mathcal V_\mathcal M$ be the odd $\mathcal C^\infty_\mathcal M$-linear operator well-defined by $\pi^2=Id$ and $\pi(\chi)(\omega):=\iota_\chi \omega$ for $\chi \in \mathcal V_M\subset\mathcal V_{\mathcal M,\bar 0} $ and $\omega \in  \mathcal C^\infty_\mathcal M$.

\bigskip\noindent
By \cite[prop. 4.1]{MS} there exist  non-degenerate $2$-forms $\eta \in \mathcal C^\infty_\mathcal M$ compatible with $J_M$. We fix one and denote by $g^\prime$ the $J_M$-invariant Riemannian metric $\eta(\cdot,J_M(\cdot))$ on $M$.  Furthermore we embed $$End(\mathcal V_M)\cong( \mathcal V_M \otimes_{\mathcal C^\infty_M(M)} \mathcal E^1_M) \hookrightarrow \mathcal V_\mathcal M\quad \mbox{by} \quad \chi \otimes \alpha \mapsto \xi_{\chi\otimes \alpha}:=\alpha\cdot \chi$$ and obtain $\xi_{Id},\xi_{J_M} \in \mathcal V_{\mathcal M, 1}$ and $\pi(\xi_{Id}),\pi(\xi_{J_M}) \in \mathcal V_{\mathcal M, 0}$. 

\bigskip\noindent
We define on $\mathcal M$ by $\mathcal C^\infty_\mathcal M$-linear continuation to $End(\mathcal V_\mathcal M)$, resp. $Hom(\mathcal V_\mathcal M,\mathcal V_\mathcal M^\ast)$:
\begin{itemize}
 \item[(i)] the split almost complex structure $J_R$ by $J_M \oplus (\pi\circ J_M\circ\pi) \in  End_{\mathcal C^\infty_M}(\mathcal V_M \oplus \pi(\mathcal V_M))$
 \item[(ii)] the split Riemannian metric $g_R$ by $g^\prime + (\eta\circ (\pi\otimes\pi)) \in Hom_{\mathcal C^\infty_M}(\mathcal V_M^{\otimes2}\oplus \pi(\mathcal V_M)^{\otimes2},\mathcal C^\infty_M)$ and $g_R(\mathcal V_M \otimes \pi(\mathcal V_M))=0$
\end{itemize}

Note the following technical lemma for later application:
\begin{lem}\label{tech}
 Let $f,g \in End(\mathcal V_M)$, $\omega \in \mathcal C^\infty_{\mathcal M,2}$. Then:
 \begin{itemize}
  \item[a)] $\quad{\pi(\xi_{f})} (\omega) =\frac{1}{2}(\omega(f(\cdot),\cdot)+\omega(\cdot,f(\cdot)))$ 
  \item[b)] $\quad[\pi(\xi_f),\pi(\xi_g)]=-\pi(\xi_{[f,g]})$
  \item[c)] $\quad J_R(\xi_{J_M})=-\xi_{Id}\quad$ and  $\quad J_R(\pi(\xi_{J_M}))=-\pi(\xi_{Id})$
 \end{itemize}
\end{lem}

Further fix in the almost complex, resp. Riemannian case the tensors:
\begin{itemize}
 \item[(i)] $J=J_R\exp(\eta \cdot Y_\eta)$ with  $Y_{\eta}\in End^{(2)}(\mathcal V_\mathcal M)_{\bar 0}$ by $Y_{\eta}=F_{J_R}(\pi(\xi_{J_M})\otimes {\rm d}_\mathcal M \eta)  $
 \item[(ii)] $g=g_R\exp(\eta \cdot W_\eta)$ with $W_\eta\in End^{(2)}(\mathcal V_\mathcal M)_{\bar 0}$ by $W_\eta=F_{g_R}(\pi(\xi_{J_M})\otimes d_\mathcal M\eta)$
\end{itemize}
We prove:
\begin{lem}\label{t2}
 Assume that $n>1$. The endomorphisms $Y_\eta$ and $W_\eta$ are nowhere vanishing.
\end{lem}
\begin{proof}
 With Lemma \ref{tech} c) follows 
  $Y_{\eta}=\pi(\xi_{J_M})\otimes {\rm d}_\mathcal M  \eta-\pi(\xi_{Id})\otimes (({\rm d}_\mathcal M  \eta)\circ J_R) $.
 Applying $(Y_\eta(\pi(\xi_{J_M})))(\eta)$ we obtain $(\pi(\xi_{J_M})(\eta))^2+(\pi(\xi_{Id})(\eta))^2$. 
By Lemma \ref{tech} a) it is $\pi(\xi_{J_M})(\eta)=0$ since $\eta$ is compatible with $J_M$, and $(Y_\eta(\pi(\xi_{J_M})))(\eta)=\eta^2$. So $Y_{\eta}$ is nowhere vanishing. 

\smallskip\noindent
For the second statement note  
 $W_\eta=\pi(\xi_{J_M})\otimes d_\mathcal M\eta+g_R^{-1}(d_\mathcal M \eta)\cdot g_R(\pi(\xi_{J_M}))$. 
 Further $(W_\eta(\pi(\xi_{J_M})))(\eta)=(\pi(\xi_{J_M})(\eta))^2+(g_R^{-1}(d_\mathcal M \eta))(\eta)\cdot \eta(\xi_{J_M},\xi_{J_M})$. Due to the compatibility of $\eta$ and $J_M$, it is $\eta(\xi_{J_M},\xi_{J_M})=\eta$ and as before, $\pi(\xi_{J_M})(\eta)=0$. Further a calculation yields $(g_R^{-1}(d_\mathcal M \eta))(\eta)=g_R(g_R^{-1}(d_\mathcal M \eta),g_R^{-1}(d_\mathcal M \eta))=\eta$ up to terms of degree $4$ and higher. Hence $(W_\eta(\pi(\xi_{J_M})))(\eta)=\eta^2$ up to terms of degree $6$ and higher. So $W_\eta$ is nowhere vanishing.
\end{proof}

Finally it follows:
\begin{theo} 
 Assume that $n>2$. The almost complex structure $J=J_R\exp(\eta \cdot Y_\eta)$ and the Riemannian metric $g=g_R\exp(\eta \cdot W_\eta)$ on $\mathcal M$ are nowhere split.
\end{theo}
\begin{proof}
For $\psi^\ast=\exp(\xi)\psi_0^\ast \in Hol(\mathcal M,J_R,4)$ we obtain $[ad(\xi_2),J_R]=0$. With the identity $\exp(\xi_{2}).J_R=\exp(ad(\xi_{2}))J_R \exp(ad(\xi_{2}))=\exp([ad(\xi_{2}),\cdot])(J_R)=J_R$ it follows that $\psi^\ast$ maps  $Y\in F_{J_R} ( End^{\{4\}}(\mathcal V_{\mathcal M})_{\bar 0})$ to   $F_{J_R}(ad(\xi_4))+\psi^\ast_0.Y$ up to terms of degree $\geq 6$. Hence $F_{J_R} (ad(\mathcal V_{\mathcal M,\bar 0}^{\{4\}}))$ is a $PHol(\mathcal M,J_R,4)$-orbit up to terms of degree $\geq 6$. Using Proposition \ref{prop1} and $\eta Y_\eta=F_{J_R}(\eta \pi(\xi_{J_M})\otimes {\rm d}_\mathcal M \eta)$ it is sufficient to check that $\eta F_{J_R}(ad(\eta \pi(\xi_{J_M})))$ does not vanish in degree $4$. From the proof of Proposition \ref{prop1} we know the identity $F_{J_R}(ad(\eta\pi(\xi_{J_M})))=\eta\cdot  F_{J_R}(ad(\pi(\xi_{J_M})))-Y_\eta$. Now Lemma \ref{tech} b) and c) yields that $ F_{J_R}(ad(\pi(\xi_{J_M})))(\pi(\xi_{J_M}))=0$. Hence it is  $F_{J_R}(ad(\eta\pi(\xi_{J_M})))(\pi(\xi_{J_M}))(\eta)=-Y_\eta(\pi(\xi_{J_M}))(\eta)$ which is $\eta^2$ as it was shown in the proof of Lemma \ref{t2}. This proves the first statement.

\smallskip\noindent
It is by direct calculation $\eta W_\eta \in End^{(4)}_{g_0}(\mathcal V_\mathcal M)_{\bar 0}$. 
 Let $\psi^\ast=\exp(\xi)\psi_0^\ast \in Iso(\mathcal M,g_R,6)$, then the degree $4$ term of $\psi^\ast.g$ vanishes while the degree $6$ term is $\psi_0^\ast.(g_2W_4+g_0W_6)+(\psi^\ast.g_R)_6$. Further $(1+\xi_2)\psi_0^\ast$ preserves $g_R$ and so does $\exp(\xi_2)\psi_0^\ast$. The term $(\exp(\xi_4+\xi_6).g_R)_6$ equals $(g_R(ad(\xi_4+\xi_6)\otimes Id + Id \otimes ad(\xi_4+\xi_6))-(\xi_4+\xi_6) g_R)_6$. So $\psi^\ast$ maps $W \in End^{(4)}_{g_0}(\mathcal V_\mathcal M)_{\bar 0}$ to $G_{g_R}(\xi_4)+\psi_0^\ast(W_4)$ up to terms of degree $\geq 6$. Since $\psi^\ast$ preserves the vanishing degree four term of $g_R$, it follows that $\xi_4 \in \mathcal V_{\mathcal M,g_0,\bar 0}^{\{4\}}$.  Hence $G_{g_R} (\mathcal V_{\mathcal M,g_0,\bar 0}^{\{4\}})$ is a $PIso(\mathcal M,g_R,6)$-orbit up to terms of degree $\geq 6$.
 Analogue to the almost complex case and following Proposition \ref{prop2} it is sufficient to show that $\eta G_{g_R}(\eta \pi(\xi_{J_M}))$ is nowhere vanishing. We have the identity $G_{g_R}(\eta \pi(\xi_{J_M}))=\eta\cdot  G_{g_R}(\pi(\xi_{J_M}))-W_\eta$. Note that for $\alpha \in \mathcal V_\mathcal M^\ast$ it is $(g_R^{-1}(\alpha))(\eta)=\alpha(g_R^{-1}(d_\mathcal M \eta))$ and $g_R^{-1}(d_\mathcal M\eta)=\pi(\xi_{Id})$ up to terms of degree two and higher. Using these details, Lemma \ref{tech} b) and the definition of $g_R$ one obtains $G_{g_R}(\pi(\xi_{J_M}))(\pi(\xi_{J_M}))(\eta)=-\pi(\xi_{J_M})(\eta (\pi(\xi_{Id}) ,\pi(\xi_{J_M})))$ up to terms of degree four and higher. A direct calculation using the graded Leibniz rule and $J_M$-invariance of $\eta$ shows that $G_{g_R}(\pi(\xi_{J_M}))(\pi(\xi_{J_M}))(\eta)$ vanishes up to terms of degree four and higher. Hence $G_{g_R}(\eta \pi(\xi_{J_M}))(\pi(\xi_{J_M}))(\eta)=-W_\eta (\pi(\xi_{J_M}))(\eta)$ up to terms of degree 6 and higher. In the proof of Lemma \ref{t2} it was shown that the degree 4 term of this expression is $\eta^2$. This proves the second statement.
\end{proof}

\end{document}